\documentclass[10pt,reqno]{amsart}

\usepackage[T1]{fontenc}
\usepackage{lmodern}
\usepackage{microtype}
\usepackage{amsmath,amssymb,mathrsfs,mathtools}
\usepackage{xcolor} 
\usepackage[colorlinks=true,linkcolor=blue!55!black,
  citecolor=blue!55!black,urlcolor=blue!55!black]{hyperref}
\usepackage{cleveref}
\usepackage{aliascnt}

\allowdisplaybreaks[3]

\newtheorem{theorem}{Theorem}[section]
\newaliascnt{proposition}{theorem}
\newtheorem{proposition}[proposition]{Proposition}
\aliascntresetthe{proposition}
\newaliascnt{lemma}{theorem}
\newtheorem{lemma}[lemma]{Lemma}
\aliascntresetthe{lemma}
\newaliascnt{corollary}{theorem}

\aliascntresetthe{corollary}

\theoremstyle{definition}

\newtheorem{example}[theorem]{Example}

\crefname{proposition}{Proposition}{Propositions}
\crefname{theorem}{Theorem}{Theorems}
\crefname{lemma}{Lemma}{Lemmas}
\crefname{corollary}{Corollary}{Corollaries}
\newaliascnt{conjecture}{theorem}
\newtheorem{conjecture}[conjecture]{Conjecture}
\aliascntresetthe{conjecture}
\crefname{conjecture}{Conjecture}{Conjectures}
\Crefname{conjecture}{Conjecture}{Conjectures}
\newcommand{\Q}{\mathbb Q}
\newcommand{\Z}{\mathbb Z}
\newcommand{\R}{\mathbb R}
\newcommand{\C}{\mathbb C}
\newcommand{\V}{V_K}
\newcommand{\dd}{\mathrm d}
\newcommand{\eps}{\varepsilon}
\newcommand{\cl}[1]{\langle #1\rangle}
\DeclareMathOperator{\Arg}{Arg}
\DeclareMathOperator{\sgn}{sgn}
\newcommand{\M}{\mathcal M}
\newcommand{\Li}{\operatorname{Li}}
\newcommand{\Norm}{\operatorname{N}}
\newcommand{\atanh}{\operatorname{arctanh}}
\DeclareMathOperator{\Log}{Log}
\numberwithin{equation}{section}

\title[The weak Chinburg conjecture on Mahler measures]{The weak Chinburg conjecture on Mahler measures}

\author[X. Guo]{Xuejun Guo}
\address{School of Mathematics, Nanjing University, Nanjing 210093,
People's Republic of China}
\email{guoxj@nju.edu.cn}

\author[Z. Tao]{Zhengyu Tao}
\address{School of Mathematics, Hefei University of Technology,
Hefei 230009, People's Republic of China}
\email{taozhy@hfut.edu.cn}
\date{}

\hypersetup{
  pdftitle={The weak Chinburg conjecture on Mahler measures},
  pdfauthor={Xuejun Guo and Zhengyu Tao}
}


\begin{document}

\begin{abstract}
For every negative fundamental discriminant $-f$ and every $k\geq1$,
we construct a rational function
$R_{f,2k}\in\Q(x_1,\ldots,x_{2k})$ and a constant
$r_{f,2k}\in\Q^\times$ such that
\[
m(R_{f,2k})=r_{f,2k}L'(\chi_{-f},1-2k),
\]
where $m$ denotes the logarithmic Mahler measure and $\chi_{-f}$ is
the quadratic Dirichlet character associated with $-f$. This proves the weak Chinburg conjecture. An independent construction using Bloch cycles yields a stronger result in two variables: there exists a nonzero polynomial
$R_f\in\Q[x,y]$ satisfying
\[
m(R_f)=4wL'(\chi_{-f},-1),
\]
where $w$ is the number of roots of unity in
$\Q(\sqrt{-f})$.
\end{abstract}

\maketitle

\section{Introduction}\label{sec:introduction}

For a nonzero rational function
$R\in\mathbb{C}(x_1,\ldots,x_d)$, its logarithmic Mahler measure
is defined by
\begin{equation}\label{eq:mahler}
 m(R)=\frac{1}{(2\pi)^d}
 \int_0^{2\pi}\cdots\int_0^{2\pi}
 \log\left|R\left(e^{it_1},\ldots,e^{it_d}\right)\right|
 \,\mathrm{d}t_1\cdots\mathrm{d}t_d.
\end{equation}  Let $-f<0$ be a fundamental discriminant, and let
$\chi_{-f}=\left(\frac{-f}{\,\cdot\,}\right)$ be the associated
primitive odd quadratic Dirichlet character.  In 1981, Smyth \cite{Smyth2} established the identity
\[
 m(1 + x_1 + x_2) = \frac{3\sqrt{3}}{4\pi} L( \chi_{-3},2) = L'(\chi_{-3}, -1)
\]
relating the Mahler measure of a bivariate polynomial to a Dirichlet \(L\)-value.

In 1984, Chinburg \cite{Chinburg} formulated a conjecture and posed a stronger question, which we refer to as the weak and strong Chinburg conjectures, respectively.

\begin{conjecture}[Chinburg]\label{con:Chinburg}  Let $-f<0$ be a fundamental discriminant and $k$ a positive integer. 
\begin{itemize}
    \item[\textup{(1)}] \textup{(weak form)} There exists a  rational function $R \in \Q(x_1,\ldots,x_{2k})^\times$ and a rational number $r \in \Q^{\times}$ such that $m(R) = r L'(\chi_{-f}, 1-2k).$
    \item[\textup{(2)}] \textup{(strong form)}  One may take a nonzero polynomial $R \in \Z[x_1,\ldots,x_{2k}]$ in \textup{(1)}.
\end{itemize}
\end{conjecture}

Deninger \cite{Deninger} related Mahler measures to Deligne periods
and regulators, while Boyd and Rodriguez-Villegas \cite{BRV} developed
their connection with dilogarithms and hyperbolic geometry.
For specific conductors, Ray \cite{Ray} obtained examples using twisted
$L$-series, and Liu and Qin \cite{LQ} reported further numerical candidates at $f=23,303,755$. Bertin and Mehrabdollahei \cite{BM} constructed
polynomial families yielding exact Mahler measure identities related
to the weak conjecture. Hokken, Mehrabdollahei, and Ringeling proved
the bivariate weak conjecture with cyclotomic coefficients and gave
further numerical candidates over $\Q$
\cite[Theorem~4.1 and Tables~2--3]{HMR}; their Table~1 also summarizes
previously known examples of the strong form.

In higher dimensions, Lal\'in \cite{Lalin06} obtained formulas for
families with arbitrarily many variables in terms of zeta values and
Dirichlet $L$-values. Lal\'in, Nair, and Roy \cite{LNR} studied families with non-linear degree dependence, while Nair \cite{Nair} obtained further formulas involving zeta values and the Dirichlet \(L\)-function of conductor \(3\).

Prior to the present work, both the weak and strong versions of Chinburg's
conjecture remained open in general; see \cite[Section~1]{BM},
\cite[Remark~4.2.5]{Pengo}, and \cite[Remark~B]{HMR}.
We prove the weak conjecture and obtain a polynomial realization over \(\Q\) in two variables.

\begin{theorem}\label{thm:weak-higher} Let \(-f\) be a negative fundamental discriminant. 
For every integer \(k\geq1\), there exist a nonzero rational function \(R_{f,2k}\in\Q(x_1,\ldots,x_{2k})\) and a nonzero rational number \(r_{f,2k}\) such that
\begin{equation}\label{eq:weak-higher}
 m(R_{f,2k})=r_{f,2k}L'(\chi_{-f},1-2k).
\end{equation}
\end{theorem}

The proof is effective in principle, although we do not attempt to
give general degree bounds for $R_{f,2k}$.

\begin{theorem}\label{thm:weak}
Let $-f$ be a negative fundamental discriminant, and let $w$ be
the number of roots of unity in $\Q(\sqrt{-f})$.
There exists a nonzero polynomial $R_f\in\Q[x,y]$ such that
\begin{equation}\label{eq:weak}
 m(R_f)=4w\,L'(\chi_{-f},-1).
\end{equation}
\end{theorem}

Theorem~\ref{thm:weak} strengthens the bivariate weak conjecture by
producing a polynomial over $\Q$. It does not by itself imply the
strong conjecture: multiplying by a positive integer $N$ to clear
denominators adds $\log N$ to the Mahler measure.

The proof of Theorem~\ref{thm:weak-higher} begins with the rational function
\[
U_f(t)=\prod_{\substack{a\bmod f\\\chi_{-f}(a)\ne0}}
       (1-\zeta_f^a t)^{\chi_{-f}(a)},
\qquad \zeta_f=e^{2\pi i/f},
\]
whose coefficients lie in $K=\Q(\sqrt{-f})$, with $\sqrt{-f}=i\sqrt f$.
Its coefficient conjugate is $U_f^{-1}$. Taking the coefficient norm
of $x-U_f(t)$, setting $t=is$, and applying Jensen's formula in $x$
gives $|\log|U_f(is)||$. This loses the sign needed in the
character sum. For $f\ne4$, we recover it by introducing
\[
V_f(t)=\frac{1-\sqrt{-f}\,t}{1+\sqrt{-f}\,t}.
\]
Writing $\ell_f(s)=\log|U_f(is)|$ and
$g_f(s)=\log|V_f(is)|$, we choose an integer $M$ large enough
that $|\ell_f(s)|<Mg_f(s)$ for positive $s$ away from the
pole. Since both functions are odd, this gives
\[
|\ell_f(s)+Mg_f(s)|-M|g_f(s)|
   =\operatorname{sgn}(s)\ell_f(s)
\]
almost everywhere. The left-hand side is the Jensen integral in one
variable of a quotient of norms defined over $\Q$.
Lal\'in's integral formulas \cite{Lalin03,Lalin06} then express
its Mahler measure in terms of the desired polylogarithmic
character sum and lower-weight terms. A second norm
construction uses weak approximation to isolate one embedding of a
suitable totally real field in Jensen's formula, allowing these
additional terms to be realized over $\Q$ and eliminated
successively. The quadratic Gauss sum and the functional
equation identify the remaining sum as a nonzero rational multiple of
$L'(\chi_{-f},1-2k)$. The case $f=4$ follows directly from
Lal\'in's formulas.

For Theorem~\ref{thm:weak}, the basic building block is the
quadratic polynomial
\[
Q_{A,B}(x,y)=(1+Ax-By)^2-4Ax,
\qquad A=|z|^2,\quad B=|1-z|^2,
\]
where $z$ is a nonreal element of an imaginary quadratic field.
Its coefficients are rational, and Maillot's formula expresses
its Mahler measure in terms of the Bloch--Wigner dilogarithm $D(z)$
and a logarithmic correction.
The key observation is that these corrections can be removed
when the corresponding integer combination of symbols is a Bloch
cycle. More precisely, the cycle
condition forces certain products of $z$ and $1-z$ to be roots
of unity. If the field contains $w$ roots of unity, the total
contribution of these corrections to the Mahler measure has the form
\[
\frac{4}{w}\sum_p k_p\log p,\qquad k_p\in\Z.
\]
Taking a power and multiplying by a rational number removes this
term. We apply this construction to the geometric cycles of
Burns, de Jeu, Gangl, Rahm, and Yasaki \cite{BDGRY}: their conjugate
pairs give positive exponents in the polynomial product, and their regulator
formula gives the required multiple of $L'(\chi_{-f},-1)$.

The examples at conductor \(31\) illustrate the difference between the two constructions: the Bloch-cycle method yields a polynomial of total degree \(132\), whereas, for the explicit parameters chosen here, the general construction yields a rational function of substantially larger degree.

Theorem~\ref{thm:weak-higher} is proved in
Section~\ref{higher:sec:general}, and 
Theorem~\ref{thm:weak} is proved in Section~\ref{sec:cycles}.
Section~\ref{sec:examples} first gives a family at conductor $8$ for every $k \ge 1$ and then compares two realizations at conductor
$31$ in the case $k=1$.

\section{Rational functions in arbitrary even dimension}
\label{higher:sec:general}
The proof of Theorem~\ref{thm:weak-higher}  combines a norm construction that preserves signed character sums with triangular identities relating Mahler measures to polylogarithmic values. We begin with the integral representation underlying these identities.

Let $\M_d$ be the $\Q$-linear span of Mahler measures of nonzero
rational functions over $\Q$ in $d$ variables. Since adjoining an unused variable does not change the Mahler measure, we have $\M_d\subseteq \M_{d+1}$. For the particular family used below, set
\[
T_d=\prod_{j=1}^d\frac{1-x_j}{1+x_j},\qquad M_d(z)=m(z+T_d).
\]
We use the half-angle substitution of \cite[Proposition 1]{Lalin03}
and \cite[Section 5]{Lalin06}. With $x_j=e^{i\theta_j}$ and
$t_j=-\tan(\theta_j/2)$, it gives
\begin{equation}\label{higher:real-integral}
 M_d(z)=\frac1{\pi^d}\int_{\R^d}
 \log|z+i^dt_1\cdots t_d|\prod_{j=1}^d\frac{\dd t_j}{1+t_j^2}.
\end{equation}
To reduce integrals over the positive orthant to one dimension,
introduce the density of the product $t_1\cdots t_d$:
\begin{equation}\label{higher:kernel-definition}
 p_1(t)=\frac2{\pi(1+t^2)},\qquad
 p_{d+1}(t)=\int_0^\infty p_d(u)p_1(t/u)\frac{\dd u}{u}
 \quad(t>0).
\end{equation}
The substitution $t=t_1\cdots t_d$, as used in
\cite[Section 5]{Lalin06}, gives in our notation
\begin{equation}\label{higher:kernel-reduction}
 \left(\frac2\pi\right)^d\int_{(0,\infty)^d}
 h(t_1\cdots t_d)\prod_{j=1}^d\frac{\dd t_j}{1+t_j^2}
 =\int_0^\infty h(t)p_d(t)\,\dd t
\end{equation}
for every measurable $h$ for which the integrals converge absolutely.
Each $p_d$ is a probability density  satisfying $p_d(t)\leq1/(\pi t)$,
since $p_1(t/u)/u\leq1/(\pi t)$. Splitting \eqref{higher:real-integral}
into orthants therefore gives
\begin{equation}\label{higher:kernel-Md}
 M_d(z)=\frac12\int_0^\infty
 p_d(t)\bigl(\log|z+i^dt|+\log|z-i^dt|\bigr)\,\dd t.
\end{equation}

We now encode the quadratic character in a rational function.
Assume $f\ne4$; the case $f=4$ is treated in the proof of
Theorem~\ref{thm:weak-higher}. Write $\chi=\chi_{-f}$, $\vartheta=i\sqrt f$, $K=\Q(\vartheta)$, and
$\zeta=e^{2\pi i/f}$. Define
\begin{equation}\label{higher:Af}
\prod_{\substack{a\bmod f\\\chi(a)=1}}(1-\zeta^at)
       =P_f(t)+\vartheta Q_f(t),\quad P_f,Q_f\in\Q[t],\qquad
U_f(t)=\frac{P_f(t)+\vartheta Q_f(t)}{P_f(t)-\vartheta Q_f(t)}.
\end{equation}

The product is fixed by $\ker\chi$, so its coefficients lie in $K$;
its conjugate contains the factors with $\chi(a)=-1$. Thus this $U_f$ agrees
with the product in the introduction.
For real $s$, set $\ell_f(s)=\log|U_f(is)|$.
Pairing $a$ and $-a$ in the logarithm of $U_f$ gives
\begin{equation}\label{higher:ell-pair}
\ell_f(s)=\sum_{\substack{0<a<f/2\\\chi(a)\ne0}}\chi(a)
 \log\left|\frac{\zeta^a+is}{\zeta^a-is}\right|
 =\sum_{\substack{0<a<f/2\\\chi(a)\ne0}}\chi(a)
 \atanh\!\left(\frac{2s\sin(2\pi a/f)}{1+s^2}\right)\quad(s>0).
\end{equation}
 For $f\ne4$, none of the relevant roots
is $\pm i$. Thus $\ell_f$ is finite on the positive half-line;
$\ell_f(-s)=-\ell_f(s)$ and $\ell_f(1/s)=\ell_f(s)$.

To see what happens when we take norms, let $u=U_f(is)$.
Since coefficient conjugation gives $\overline{U_f}(t)=U_f(t)^{-1}$,
the norm of $x-U_f(t)$ specializes at $t=is$ to
$(x-u)(x-u^{-1})$. Jensen's formula gives
\[
 m\bigl((x-u)(x-u^{-1})\bigr)
 =\log\max\{1,|u|\}+\log\max\{1,|u|^{-1}\}
 =|\ell_f(s)|.
\]
To recover the sign lost in this absolute value, introduce the
auxiliary function
\begin{equation}\label{higher:reference}
V_f(t)=\frac{1-\vartheta t}{1+\vartheta t},\qquad
g_f(s)=\log|V_f(is)|=
\log\left|\frac{1+\sqrt f\,s}{1-\sqrt f\,s}\right|.
\end{equation}
The function $g_f$ is odd on
$\R\setminus\{\pm1/\sqrt f\}$ and strictly positive for
$s>0$, $s\ne1/\sqrt f$, with logarithmic singularities at
$s=\pm1/\sqrt f$.

For $s>0$, $s\ne1/\sqrt f$, we have
$g_f(s)=2\atanh\min\{\sqrt f\,s,(\sqrt f\,s)^{-1}\}$, so
$g_f(s)\ge2s/(\sqrt f(1+s^2))$.
Together with $\atanh u\le u/(1-u^2)$ for $0\le u<1$ and
$1-(2s\sin\alpha/(1+s^2))^2\ge\cos^2\alpha$, this gives
\[
 \frac{|\ell_f(s)|}{g_f(s)}\le
 \sqrt f\sum_{\substack{0<a<f/2\\\chi(a)\ne0}}
 \frac{\sin(2\pi a/f)}{\cos^2(2\pi a/f)}
 \qquad(s>0,\ s\ne1/\sqrt f).
\]
Choose a positive integer $M$ exceeding this finite algebraic bound.
Then $|\ell_f(s)|<Mg_f(s)$ on this set, so almost everywhere on $\R$,
\begin{equation}\label{higher:signed-identity}
|\ell_f(s)+Mg_f(s)|-M|g_f(s)|=\sgn(s)\ell_f(s).
\end{equation}
Set $W_f=U_fV_f^M$.  Jensen's formula gives
\begin{equation}\label{higher:jensen-signed}
\begin{aligned}
m_x\left(
 \frac{(x-W_f(is))(x-W_f(is)^{-1})}
      {\bigl((x-V_f(is))(x-V_f(is)^{-1})\bigr)^M}
\right)
&=|\ell_f(s)+Mg_f(s)|-M|g_f(s)|\\
&=\sgn(s)\ell_f(s)=\ell_f(|s|),
\end{aligned}
\end{equation}
where $m_x$ denotes Mahler measure in $x$ alone.
Since coefficient conjugation sends $W_f$ to $W_f^{-1}$,
$(x-W_f(is))(x-W_f(is)^{-1})$ is the specialization at $t=is$
of $\Norm_{K/\Q}(x-W_f(t))$, and similarly for $V_f$.

Let $n=2k$ and $d=n-1$. Define
\begin{equation}\label{higher:R-def}
\mathcal R_{f,n}=
\frac{\Norm_{K/\Q}(x_n-W_f(T_d))}
{\Norm_{K/\Q}(x_n-V_f(T_d))^M}
\in\Q(x_1,\ldots,x_n)^\times.
\end{equation}

\begin{proposition}\label{higher:lift}
For every $f\ne4$ and every even $n\ge2$,
\begin{equation}\label{higher:lift-eq}
m(\mathcal R_{f,n})=
2\sum_{0<a<f/2}\chi(a)
\bigl(M_n(\zeta^a)-M_{n-1}(\zeta^a)\bigr).
\end{equation}
\end{proposition}
\begin{proof}
Since $d=n-1$ is odd, write $T_d=is$ on the unit torus, with
$|s|=\prod_{j=1}^d|t_j|$ in the coordinates of \eqref{higher:real-integral}.
By \eqref{higher:jensen-signed}, integration in $x_n$ gives
$\ell_f(|s|)$. Integrating the remaining variables by
\eqref{higher:kernel-reduction} therefore yields
\[
 m(\mathcal R_{f,n})=\int_0^\infty\ell_f(t)p_d(t)\,\dd t.
\]
For $\operatorname{Im}z>0$, Poisson's formula reads
\[
 \frac1\pi\int_{\R}\frac{\log|z-uv|}{1+v^2}\,\dd v
 =\log\bigl|z+i\,|u|\bigr|\qquad(u\in\R),
\]
where the sign of $u$ is removed by $v\mapsto-v$.
Since $n$ is even, applying Poisson's formula to $t_n$ in
\eqref{higher:real-integral}, with
$u=-i^n t_1\cdots t_d\in\R$, and then using
\eqref{higher:kernel-reduction} gives
\[
 M_n(z)=\int_0^\infty p_d(t)\log|z+it|\,\dd t.
\]
Since $d=n-1$ is odd, $i^d=\pm i$, so
\eqref{higher:kernel-Md} gives
\[
 M_{n-1}(z)=\frac12\int_0^\infty p_d(t)
 \bigl(\log|z+it|+\log|z-it|\bigr)\,\dd t.
\]
Subtracting, we obtain
\[
 M_n(z)-M_{n-1}(z)
 =\frac12\int_0^\infty p_d(t)
 \log\left|\frac{z+it}{z-it}\right|\,\dd t.
\]
Setting $z=\zeta^a$, multiplying by $2\chi(a)$, and summing over
$0<a<f/2$ with $\chi(a)\ne0$ proves the claim by
\eqref{higher:ell-pair}.
\end{proof}

The terms \(M_{n-1}(\zeta^a)\) in \eqref{higher:lift-eq} are Mahler measures of functions with algebraic coefficients. The following norm construction shows that these values belong to \(\mathcal M_n\). 
Let $E$ be a totally real number field and
$B\in E(x_1,\ldots,x_d)^\times$, with
$d\ge0$, and suppose $0<b_\sigma\le|B^\sigma|\le C_\sigma<\infty$
almost everywhere on the unit torus. For a chosen embedding $\sigma_0$,
weak approximation provides $a\in E$ satisfying
\[
 a^{\sigma_0}>\max\{1,b_{\sigma_0}^{-1}\},\qquad
 0<a^\sigma<\min\{1,C_\sigma^{-1}\}\quad(\sigma\ne\sigma_0).
\]
Jensen's formula in $v$ gives
\begin{equation}\label{higher:bounded-descent}
 m\left(\frac{\Norm_{E/\Q}(v-aB)}{\Norm_{E/\Q}(v-a)}\right)
   =m(B^{\sigma_0})\in\M_{d+1}.
\end{equation}
Only $\sigma_0$ contributes to either norm, and the quotient cancels
the term $\log a^{\sigma_0}$.
We apply this construction to a bounded quotient. For a nonreal
root of unity $z\ne\pm i$, let
$c=(z+z^{-1})/2\in E=\Q(z+z^{-1})$, so that
$m(T_d^2+2cT_d+1)=2M_d(z)$ by coefficient conjugation.
Choose $\sigma_0$ to be the inclusion $E\subset\R$.
To obtain a quotient bounded above and away from zero at every
embedding, set $B=(T_d^2+2cT_d+1)/(1+(-1)^dT_d^2)$.
The moduli $|B^\sigma|$ lie almost everywhere
between $1-|c^\sigma|$ and $1+|c^\sigma|$ for even $d$, and between
$|c^\sigma|$ and $1$ for odd $d$, with positive lower bounds since
$0<|c^\sigma|<1$; the poles of $T_d$ form a set of measure zero.
By \eqref{higher:bounded-descent}, $m(B)\in\M_{d+1}$, while
$m(1+(-1)^dT_d^2)\in\M_d\subseteq\M_{d+1}$.
Adding these measures gives $2M_d(z)$, hence
\begin{equation}\label{higher:Md-descent}
M_d(z)\in\M_{d+1}\qquad(d\ge1, z\neq\pm i).
\end{equation}

We next use Lalín's kernel formulas to relate \(M_d(z)\) to normalized polylogarithmic values. 
For $n\geq1$, set $A_n(z)=\Li_n(z)-\Li_n(-z)$. At a nonreal root
of unity define the normalized periods
\[
B_n(z)=\begin{cases}
\pi^{1-n}\operatorname{Im}A_n(z),&n\text{ even},\\
\pi^{1-n}\operatorname{Re}A_n(z),&n\text{ odd}.
\end{cases}
\]
We first establish identities on half-planes, then specialize to
roots of unity.
\begin{proposition}\label{higher:triangular}
For each integer $r\geq1$ there are rational numbers $c_{rj}$ and
$d_{rj}$, $1\leq j\leq r$, such that the following identities hold:
\eqref{higher:even-triangle} for $\operatorname{Im}z>0$, and
\eqref{higher:odd-triangle} for $\operatorname{Re}z>0$.
\begin{align}
M_{2r}(z)&=\tfrac12\log|1-z^2|
 +\sum_{j=1}^r c_{rj}\pi^{1-2j}\operatorname{Im}\mathcal V_{2j}(z),
 &&c_{rr}=2^{2r-2},\label{higher:even-triangle}\\
M_{2r+1}(z)&=\log|1+z|
 +\sum_{j=1}^r d_{rj}\pi^{-2j}\operatorname{Re}\mathcal U_{2j+1}(z),
 &&d_{rr}=2^{2r-1},\label{higher:odd-triangle}
\end{align}
where 
\begin{align*}
\mathcal V_n(z)&=\sum_{h=0}^{n-1}\frac{(-\Log(-iz))^h}{h!}A_{n-h}(z),
 &&\operatorname{Im}z>0,\\
\mathcal U_n(z)&=\sum_{h=0}^{n-1}\frac{(-\Log z)^h}{h!}A_{n-h}(z),
 &&\operatorname{Re}z>0.
\end{align*}
For \(\mathcal U_n\), the jumps of the polylogarithms across the branch cut \([1,\infty)\) cancel when \(n\ge2\), so the combination extends analytically across the cut. Here $\Log$ is
the principal logarithm, and the polylogarithms are continued from the
unit disk. The displayed values of $c_{rr}$ and $d_{rr}$ show that
both systems have nonzero diagonal coefficients.
\end{proposition}
\begin{proof}
We use Lal\'in's formulas for the kernels $p_d$ and derive the
triangular identities in the normalization needed here.
Specifically, \cite[Section 5, Theorem 17]{Lalin06}, with the
normalization in \eqref{higher:kernel-definition}, gives
\begin{align*}
p_{2r-1}(t)&=\frac{E_r(\log t/\pi)}{\pi(1+t^2)},&
E_r(X)&=\frac{2^{2r-1}}{(2r-2)!}
\prod_{a=1}^{r-1}\left(X^2+(a-\tfrac12)^2\right),\\
p_{2r}(t)&=\frac{O_r(\log t/\pi)}{\pi(t^2-1)},&
O_r(X)&=\frac{2^{2r}}{(2r-1)!}X
\prod_{a=1}^{r-1}(X^2+a^2).
\end{align*}
The apparent singularity in the second expression at \(t=1\) is removable. These formulas imply
that all logarithmic moments converge absolutely and that
$p_m(1/t)=t^2p_m(t)$; in particular,
\begin{equation}\label{higher:kernel-log}
 \int_0^\infty p_m(t)\log t\,\dd t=0.
\end{equation}

We prove the identities by comparing derivatives. Set
\[
 F_r(q)=\int_0^\infty p_{2r-1}(t)\Log(q+t)\,\dd t,\qquad
 G_r(z)=\int_0^\infty p_{2r}(t)\Log(z+t)\,\dd t.
\]
Poisson's formula gives $M_{2r}(z)=\operatorname{Re}F_r(-iz)$
and $M_{2r+1}(z)=\operatorname{Re}G_r(z)$ in the stated half-planes.
Define the rational coefficients by
\begin{equation}\label{higher:coefficients}
 c_{rj}=\left.\frac12E_r(\partial_u)
              \frac{u^{2j-1}}{\sin u}\right|_{u=0},\qquad
 d_{rj}=\left.\frac12O_r(\partial_u)
              \frac{u^{2j}}{\sin u}\right|_{u=0}.
\end{equation}
Here $\partial_u=d/du$; both quotients are analytic at zero.

To generate the powers of $\log t$ in the kernels, we differentiate
with respect to $s$. For $\operatorname{Re}q,\operatorname{Re}z>0$,
partial fractions and the beta integral give
\begin{align*}
 \int_0^\infty\frac{t^s\,\dd t}{(t+q)(1+t^2)}
 &=\frac\pi{1+q^2}\left(-\frac{q^s}{\sin\pi s}
   +\frac{q}{2\cos(\pi s/2)}+\frac1{2\sin(\pi s/2)}\right),\\
 \operatorname{PV}\int_0^\infty\frac{t^s\,\dd t}{(t+z)(t^2-1)}
 &=\frac\pi{z^2-1}\left(-\frac{z^s}{\sin\pi s}
   +\frac z2\tan(\pi s/2)+\frac12\cot(\pi s/2)\right).
\end{align*}
Here $\operatorname{PV}$ denotes the Cauchy principal value.
The first identity holds initially for $-1<\operatorname{Re}s<0$;
both right-hand sides have removable singularities at $s=0$.
Near $t=1$, subtracting $1/[2(1+z)(t-1)]$ from the second integrand
leaves a bounded remainder, uniformly for $s$ near zero and $z$ in compact
subsets of the right half-plane. Its positive-order $s$-derivatives are
locally integrable, since $(\log t)^j/(t^2-1)=O(|t-1|^{j-1})$ for $j\ge1$;
the tails at zero and infinity are uniformly integrable as well.
Thus differentiation under the principal-value integral is valid.
Applying $\pi^{-1}E_r(\pi^{-1}\partial_s)$ and
$\pi^{-1}O_r(\pi^{-1}\partial_s)$, respectively, and evaluating
at $s=0$ yields
\begin{align*}
 F_r'(q)&=\frac1{1+q^2}
 \left(q-2\sum_{j=1}^r\frac{c_{rj}}{(2j-1)!}
                  \left(\frac{\Log q}{\pi}\right)^{2j-1}\right),\\
 G_r'(z)&=\frac1{1+z}
 +\frac2{1-z^2}\sum_{j=1}^r\frac{d_{rj}}{(2j)!}
                  \left(\frac{\Log z}{\pi}\right)^{2j}.
\end{align*}
The elementary terms $q/(1+q^2)$ and $1/(1+z)$ arise from
the same differentiation of the trigonometric terms in the integral
identities. The apparent singularity of $G_r'(z)$ at $z=1$ is removable.

On the other hand, differentiation of the finite polylogarithmic sums gives
\[
 \mathcal V_n'(z)=\frac{2(-\Log(-iz))^{n-1}}{(n-1)!(1-z^2)},\qquad
 \mathcal U_{2j+1}'(z)=\frac{2(\Log z)^{2j}}{(2j)!(1-z^2)}.
\]
For the odd-dimensional identity, we first check continuation
across the cut in the right half-plane. For $x>1$, the upper-minus-lower jump of $\Li_m(x)$ is
$2\pi i(\log x)^{m-1}/(m-1)!$, while $\Li_m(-z)$ is analytic there.
Thus the jump of $\mathcal U_n$ is
$2\pi i(\log x)^{n-1}\sum_{h=0}^{n-1}(-1)^h/[h!(n-1-h)!]=0$
for $n\ge2$; the derivative formula also shows that the singularity
at $1$ is removable. Comparing derivatives determines the two identities up to additive constants. These constants are fixed by taking limits at zero and using \eqref{higher:kernel-log}. 
The logarithmic moment bounds justify differentiation on compact
subsets and the limits at zero. Finally, the leading terms of $E_r$
and $O_r$ in \eqref{higher:coefficients} give
$c_{rr}=2^{2r-2}$ and $d_{rr}=2^{2r-1}$.
\end{proof}

At roots of unity, the logarithmic factors in these triangular identities are rational multiples of \(i\pi\). Combining the identities with the preceding descent argument gives the following lemma.
\begin{lemma}\label{higher:low-weight}
For every nonreal root of unity $z$ and every $n\geq1$,
$B_n(z)\in\M_{n+1}$. For even $n$, we also have
$B_n(\pm i)\in\M_n$ and, for every nonreal root of unity $z$
with $\operatorname{Im}z>0$,
\begin{equation}\label{higher:congruence}
 M_n(z)\equiv2^{n-2}B_n(z)\pmod{\M_n}.
\end{equation}
\end{lemma}
\begin{proof}
Let $E=\Q(z+z^{-1})$. The square of the modulus of each nonzero
element $u$ of $\Q(z)$ is totally positive in $E$. Applying
\eqref{higher:bounded-descent} to $|u|^2$ with $d=0$ and using
$\log|u|=\tfrac12\log|u|^2$ gives $\log|u|\in\M_1$. In particular,
$B_1(z)=\log|(1+z)/(1-z)|$ belongs to $\M_1$, as do the
elementary logarithmic terms in Proposition~\ref{higher:triangular}.

Suppose first that $z\neq\pm i$. In the upper half-plane,
$\Log(-iz)/(i\pi)\in\Q$; in the right half-plane,
$\Log z/(i\pi)\in\Q$. Hence, for even $m$,
$\pi^{1-m}\operatorname{Im}\mathcal V_m(z)$ equals $B_m(z)$ plus a rational
linear combination of $B_j(z)$ with $j<m$. For odd $m\geq3$, the same
assertion holds with
$\pi^{1-m}\operatorname{Re}\mathcal U_m(z)$ in the right half-plane.
In fact, each power of the logarithm supplies the necessary power of
$\pi$, and multiplication by a power of $i$ selects the real or
imaginary part prescribed in $B_j$.

For even $n$, conjugation reduces the claim to the upper half-plane;
for odd $n$, $A_n(-z)=-A_n(z)$ reduces it to the right half-plane.
We now induct on $n$. Equation~\eqref{higher:Md-descent} places $M_n(z)$
in $\M_{n+1}$, while the induction hypothesis places every lower-weight term  $B_j(z)$ in
$\M_{j+1}\subseteq\M_n$. The nonzero rational diagonal
coefficient in \eqref{higher:even-triangle} or
\eqref{higher:odd-triangle} therefore gives
$B_n(z)\in\M_{n+1}$ for all roots under consideration.

For $z=\pm i$, every odd $B_n(z)$ vanishes, while
$B_{2r}(i)=2L(\chi_{-4},2r)/\pi^{2r-1}$.
Lal\'in's triangular formula \cite[Theorem 1(i), equation (2)]{Lalin06}
gives $B_{2r}(\pm i)\in\M_{2r}$ by elimination of lower weights.
Finally, in even weight $n$,
all terms in \eqref{higher:even-triangle} except
$2^{n-2}B_n(z)$ belong to $\M_n$ by the preceding argument.
This proves \eqref{higher:congruence}.
\end{proof}

\begin{proof}[Proof of Theorem~\ref{thm:weak-higher}]
Let $n=2k$ and $\chi=\chi_{-f}$.

For $f\ne4$, Proposition~\ref{higher:lift} gives
\begin{equation}\label{higher:S-realized}
\begin{aligned}
S_{f,n}&:=\sum_{0<a<f/2}\chi(a)M_n(\zeta^a)\\
&=\tfrac12m(\mathcal R_{f,n})
 +\sum_{0<a<f/2}\chi(a)M_{n-1}(\zeta^a)\in\M_n,
\end{aligned}
\end{equation}
where the last inclusion follows from \eqref{higher:Md-descent}.
The quadratic Gauss sum $\sum_a\chi(a)\zeta^{ah}=i\sqrt f\chi(h)$,
valid for every integer $h$, implies
\[
\sum_{a\bmod f}\chi(a)A_n(\zeta^a)
=2i\sqrt f\,(1-\chi(2)2^{-n})L(\chi,n).
\]
The defining series are absolutely convergent here. Pairing conjugate roots,
and letting $\rho_{f,n}=\sqrt f\,L(\chi,n)/\pi^{n-1}$, we obtain
\[
\sum_{0<a<f/2}\chi(a)B_n(\zeta^a)
=(1-\chi(2)2^{-n})\rho_{f,n}.
\]
Thus Lemma~\ref{higher:low-weight} and \eqref{higher:S-realized} show
$
2^{n-2}(1-\chi(2)2^{-n})\rho_{f,n}\in\M_n.
$
The coefficient \(2^{n-2}(1-\chi(2)2^{-n})\) is a nonzero rational number. This proves
$\rho_{f,n}\in\M_n$ for $f\ne4$.
For $f=4$, the identity
$A_n(i)=2iL(\chi_{-4},n)$ gives
$\rho_{4,n}=B_n(i)\in\M_n$ by
Lemma~\ref{higher:low-weight}.

The functional equation for the primitive odd quadratic character
$\chi$ \cite[Section 25.15]{DLMF} gives
\begin{equation}\label{higher:functional}
L'(\chi,1-2k)=
\frac{(-1)^{k-1}(2k-1)!f^{2k-1}}{2^{2k}}
\frac{\sqrt f\,L(\chi,2k)}{\pi^{2k-1}}.
\end{equation}
Write $\rho_{f,n}=\sum_{j=1}^s a_jm(\Phi_j)$ with $a_j\in\Q$ and
$\Phi_j\in\Q(x_1,\ldots,x_n)^\times$, adjoining unused variables if necessary.
Choose $N\in\Z_{>0}$ such that every $Na_j$ is an integer, and let
\[
 R_{f,n}=\prod_{j=1}^s \Phi_j^{Na_j}\in\Q(x_1,\ldots,x_n)^\times.
\]
Then $m(R_{f,n})=N\rho_{f,n}$, so \eqref{higher:functional} gives
the required identity with
$r_{f,2k}=(-1)^{k-1}2^{2k}N/
((2k-1)!f^{2k-1})\in\Q^\times$.
\end{proof}

\section{Bivariate polynomial realizations via Bloch cycles}\label{sec:cycles}

We now construct a rational function whose Mahler measure is an explicit multiple of the regulator of a Bloch cycle. Under a sign condition on the cycle, this function is a polynomial. The geometric cycles of \cite{BDGRY} satisfy this condition and yield Theorem~\ref{thm:weak}.

The Bloch--Wigner dilogarithm is
\[
  D(z)=\operatorname{Im}\operatorname{Li}_2(z)
        +\log|z|\Arg(1-z),
\]
with its usual continuous extension. Throughout, $\Arg$ denotes the
principal argument in $(-\pi,\pi]$. We use the standard identities
$D(\bar z)=-D(z)$ and $D(z)=0$ for real $z$, and the fact that
$D(z)>0$ in the upper half-plane. These conventions agree with
\cite[Section 3.1.2]{BDGRY}.

Fix an imaginary quadratic field $K\subset\C$. Let $\mu(K)$ denote
the group of roots of unity in $K$, and let
\[
  w=|\mu(K)|=
  \begin{cases}
    6,& K=\Q(\sqrt{-3}),\\
    4,& K=\Q(\sqrt{-1}),\\
    2,& \text{otherwise}.
  \end{cases}
\]
We write
\[
  \V=K^\times\otimes_{\Z}\Q,\qquad \cl{u}=u\otimes1.
\]
Let $\Z[K\setminus\{0,1\}]$ be the free abelian group on the formal
symbols $[z]$. For a chain
$\beta=\sum_{i=1}^r n_i[z_i]\in\Z[K\setminus\{0,1\}]$, define
\[
  \delta\beta=\sum_{i=1}^r n_i\cl{z_i}\wedge\cl{1-z_i}
       \in\bigwedge^2_{\Q}\V,
  \qquad D(\beta)=\sum_{i=1}^r n_iD(z_i).
\]
We call $\beta$ a Bloch cycle if $\delta\beta=0$.
We work with representatives at the level of chains; no quotient
by the five-term relations is needed for the realization theorem.

The following family of quadratic polynomials will serve as the building blocks of the construction. For \(A,B\in\Q_{>0}\), set
\[
  Q_{A,B}(x,y)=(1+Ax-By)^2-4Ax\in\Q[x,y].
\]

\begin{theorem}\label{thm:realization}
Let $\beta=\sum_i n_i[z_i]$ be a Bloch cycle over $K$
with every $z_i\notin\R$. Set
\[
  A_i=|z_i|^2,\qquad B_i=|1-z_i|^2,\qquad
  \eps_i=\sgn(\operatorname{Im}z_i),\qquad
  H_\beta=\prod_i Q_{A_i,B_i}^{\,n_i\eps_i}.
\]
Then $A_i,B_i\in\Q_{>0}$ and $H_\beta\in\Q(x,y)^\times$.
There is an explicitly computable $c_\beta\in\Q_{>0}$ such that
\begin{equation}\label{eq:realization}
  R_\beta=c_\beta H_\beta^w,
  \qquad m(R_\beta)=\frac{4w}{\pi}D(\beta).
\end{equation}
If $n_i\eps_i\geq0$ for every $i$, then $R_\beta\in\Q[x,y]$.
\end{theorem}

We first identify the logarithmic correction in Maillot's formula,
then use the cycle condition to remove it.
For a nonreal $z\in K$, write
\begin{equation}\label{eq:local-data}
\begin{gathered}
  \eps_z=\sgn(\operatorname{Im}z),\qquad
  C(z)=\Arg z\log|1-z|-\Arg(1-z)\log|z|.
\end{gathered}
\end{equation}

Maillot's triangle formula \cite[Proposition 7.3.1]{Maillot}, in the
form recorded in \cite[equation (31)]{Lalin03}, gives
\[
  \eps_z\,m(1+|z|x+|1-z|y)=\frac{D(z)+C(z)}{\pi}.
\]
To pass to rational coefficients, use the factorization
\[
  Q_{A,B}(x^2,y^2)
     =\prod_{s,t\in\{1,-1\}}(1+s\sqrt A\,x+t\sqrt B\,y).
\]
Together with invariance under sign changes and substitutions
$x\mapsto x^2$, $y\mapsto y^2$, this gives $m(Q_{A,B})=4m(1+\sqrt A\,x+\sqrt B\,y)$.
Hence,
\[
  \eps_z\,m(Q_{|z|^2,|1-z|^2})
       =\frac4\pi\bigl(D(z)+C(z)\bigr).
\]
So for the chain in Theorem~\ref{thm:realization},
\begin{equation}\label{eq:H-measure}
  m(H_\beta)=\frac4\pi D(\beta)+\frac4\pi C(\beta),
  \qquad C(\beta)=\sum_i n_iC(z_i).
\end{equation}
It remains to remove $C(\beta)$ by a rational scalar after taking a
suitable power.

To express $C(\beta)$ in terms of logarithms of primes, let $v_p$
be the usual exponent of a rational prime $p$
in a nonzero rational number, and define
\[
  \nu_p(u)=v_p(N_{K/\Q}(u)),\qquad u\in K^\times.
\]
This homomorphism extends to a $\Q$-linear functional on $\V$.
The cycle condition forces the corresponding products below to be
roots of unity.

\begin{lemma}\label{lem:torsion}
Let $\beta=\sum_i n_i[z_i]$ be a cycle.
For every rational prime $p$, the element
\begin{equation}\label{eq:U}
  U_p=\prod_i
       z_i^{\,n_i\nu_p(1-z_i)}
       (1-z_i)^{-n_i\nu_p(z_i)}
       \in K^\times
\end{equation}
is a root of unity. In particular, $U_p^w=1$.
\end{lemma}

\begin{proof}
Apply the contraction
\[
  \iota_p:\bigwedge^2_{\Q}\V\longrightarrow\V,
  \qquad \iota_p(a\wedge b)=\nu_p(b)a-\nu_p(a)b
\]
to $\delta\beta=0$. This gives
\[
  0=\sum_i n_i\bigl(
       \nu_p(1-z_i)\cl{z_i}-\nu_p(z_i)\cl{1-z_i}\bigr)
    =\cl{U_p}.
\]
The kernel of $K^\times\to\V$ is $\mu(K)$, so $U_p^w=1$.
\end{proof}

\begin{lemma}
\label{lem:correction}
Under the hypotheses of Theorem~\ref{thm:realization}, set
\begin{equation}\label{eq:T}
  T_p=\sum_i n_i\bigl(
       \nu_p(1-z_i)\Arg z_i-\nu_p(z_i)\Arg(1-z_i)\bigr).
\end{equation}
Then $k_p=wT_p/(2\pi)$ is an integer and
\begin{equation}\label{eq:correction}
  C(\beta)=\frac{\pi}{w}\sum_p k_p\log p.
\end{equation}
\end{lemma}

\begin{proof}
By \eqref{eq:U} and \eqref{eq:T},
$
  \exp(iT_p)=\frac{U_p}{|U_p|}.
$
Lemma~\ref{lem:torsion} implies $|U_p|=1$ and $U_p^w=1$.
It follows that $\exp(iwT_p)=1$, or $wT_p\in2\pi\Z$, as required.

Since $N_{K/\Q}(u)=|u|^2$, we have  
\begin{equation}\label{eq:log-norm}
  \log|u|=\frac12\sum_p\nu_p(u)\log p.
\end{equation}
Only finitely many primes occur for the finitely many elements
$z_i$ and $1-z_i$, so all the following sums are finite.
Substituting \eqref{eq:log-norm} into \eqref{eq:local-data} gives
\[
  C(\beta)=\frac12\sum_p T_p\log p
          =\frac\pi w\sum_p k_p\log p.
\]
\end{proof}

\begin{proof}[Proof of Theorem~\ref{thm:realization}]
Equations \eqref{eq:H-measure} and \eqref{eq:correction} give
\[
  m(H_\beta)=\frac4\pi D(\beta)+\frac4w\sum_p k_p\log p.
\]
Define
\begin{equation}\label{eq:scalar}
  c_\beta=\prod_p p^{-4k_p}\in\Q,\qquad
  R_\beta=\left(\prod_p p^{-4k_p}\right)
          \prod_i Q_{A_i,B_i}^{\,w n_i\eps_i}.
\end{equation}
By additivity of Mahler measure,
\[
  m(R_\beta)
   =-4\sum_p k_p\log p+w\,m(H_\beta)
   =\frac{4w}\pi D(\beta).
\]
If every $n_i\eps_i\geq0$, then $R_\beta\in\Q[x,y]$.

\end{proof}

The integers $k_p$ can be computed with certified error bounds.
First factor the rational norms $|z_i|^2$ and $|1-z_i|^2$ to
determine the finite set of relevant primes and the valuations in
\eqref{eq:T}. Then enclose the principal arguments and $\pi$ in
rational intervals, refining them until the resulting interval for
$wT_p/(2\pi)$ has length less than $1$. Since this number is an
integer by Lemma~\ref{lem:correction}, it is the unique integer in
that interval. This procedure terminates because every $z_i$ and
$1-z_i$ is nonreal, so their principal arguments admit arbitrarily
accurate certified approximations.

To apply the realization theorem, we use the following geometric
cycles from \cite[Theorem 4.7(i), Remark 4.9, Lemma 4.4, and
Corollary A.5]{BDGRY}. We state their properties in our notation.
\begin{proposition}[Burns--de Jeu--Gangl--Rahm--Yasaki]
\label{prop:geometric}
For every imaginary quadratic field $K\subset\C$, there are finitely
many $z_j\in K$ with $\operatorname{Im}z_j>0$ and positive integers
$c_j$ such that
\begin{equation}\label{eq:gamma}
  \gamma_K=\sum_j c_j\bigl([z_j]-[\bar z_j]\bigr),\qquad
  \delta\gamma_K=0,
\end{equation}
and
\[
  \sum_j c_jD(z_j)=-24\pi\zeta_K'(-1).
\]
\end{proposition}

Here $z_j$ are the cross-ratios of positively oriented tetrahedra in
the tessellation of \cite[Section 4.2]{BDGRY}. The coefficients $c_j$
are $24$ divided by the corresponding polytope stabilizer orders;
these orders divide $24$ by \cite[Corollary A.5]{BDGRY}.
The geometric class $\beta_{\rm geo}$ of
\cite[Theorem 4.7(i)]{BDGRY} includes a
correction $\beta_{\Q}$ represented by rational symbols. These
symbols are fixed by complex conjugation, so the correction disappears
upon antisymmetrization. By \cite[Remark 4.9]{BDGRY}, the resulting
chain $\gamma_K$ represents $2\beta_{\rm geo}$.

\begin{proof}[Proof of Theorem~\ref{thm:weak}]
Let $K=\Q(\sqrt{-f})$, write $d_f=L'(\chi_{-f},-1)$, and use the
cycle $\gamma_K$ of Proposition~\ref{prop:geometric} in
Theorem~\ref{thm:realization}. Since
$\zeta_K(s)=\zeta(s)L(\chi_{-f},s)$ and $L(\chi_{-f},-1)=0$,
\[
  \zeta_K'(-1)=\zeta(-1)L'(\chi_{-f},-1)=-\frac1{12}d_f.
\]
Using $D(\bar z_j)=-D(z_j)$, we obtain
\begin{equation}\label{eq:gamma-regulator}
  D(\gamma_K)=2\sum_j c_jD(z_j)
        =-48\pi\zeta_K'(-1)=4\pi d_f.
\end{equation}
The conjugate symbols have opposite
coefficients and imaginary parts, so their factors combine to give
\[
  H_{\gamma_K}
    =\prod_j Q_{|z_j|^2,|1-z_j|^2}^{\,2c_j}\in\Q[x,y].
\]
With the integers $k_p$ computed from $\gamma_K$ as in
Lemma~\ref{lem:correction}, define
\begin{equation}\label{eq:final-polynomial}
  R_f(x,y)=\left(\prod_p p^{-k_p}\right)
       \prod_j Q_{|z_j|^2,|1-z_j|^2}(x,y)^{\,wc_j/2}.
\end{equation}
Since $w$ is even, the exponents $wc_j/2$ are positive integers, so
this is a nonzero polynomial in $\Q[x,y]$. Moreover, \eqref{eq:scalar} gives
\[
  R_f^4=R_{\gamma_K}
   =\left(\prod_p p^{-4k_p}\right)
      \prod_j Q_{|z_j|^2,|1-z_j|^2}^{\,2wc_j}.
\]
Equations \eqref{eq:realization} and \eqref{eq:gamma-regulator} now give
\[
  m(R_f)=\frac14m(R_{\gamma_K})
        =\frac{w}{\pi}D(\gamma_K)=4w\,d_f.
\]
\end{proof}

\section{Explicit examples}\label{sec:examples}

\subsection{A family at conductor 8}

At conductor $8$, a family in all even dimensions can be built from
the following polynomials. For $d\geq1$, let
\begin{equation}\label{eq:F8}
 F_d(x_1,\ldots,x_d)=\frac12\left(
       \prod_{\nu=1}^d(1+x_\nu)^4+
       \prod_{\nu=1}^d(1-x_\nu)^4\right)\in\Z[x_1,\ldots,x_d].
\end{equation}
The coefficients are integral because the odd-total-degree terms
cancel and the even-total-degree terms occur twice.

The exponents below are obtained from a generating function that
eliminates the lower-weight terms.
\begin{proposition}\label{prop:conductor8}
For $k\geq1$ and $1\leq d\leq2k$, define
\begin{equation}\label{eq:a8}
 a_{k,d}=\left.\frac{d^{\,2k-1}}{dt^{2k-1}}
 \bigl(\cos(2t)(\cos t-\sin t)\sin^{d-1}(2t)\bigr)\right|_{t=0}
 \in\Z.
\end{equation}
Regarding $F_d$ as a polynomial in the first $d$ of the variables
$x_1,\ldots,x_{2k}$, set
\[
 R_{8,2k}=\prod_{d=1}^{2k}F_d^{a_{k,d}}
 \in\Q(x_1,\ldots,x_{2k})^\times.
\]
Then
\begin{equation}\label{eq:R8-measure}
 m(R_{8,2k})=
 \frac{2^{4k-1}(2k-1)!\sqrt2}{\pi^{2k-1}}L(\chi_{-8},2k)
 =2(-1)^{k-1}L'(\chi_{-8},1-2k).
\end{equation}
\end{proposition}
\begin{proof}
Leibniz's rule gives $a_{k,d}\in\Z$ and
$a_{k,2k}=2^{2k-1}(2k-1)!$.
Let $z_0=e^{\pi i/4}$ and $\mu_d=m(F_d)$. Since
$F_d=\tfrac12\prod_\nu(1+x_\nu)^4(1+T_d^4)$, factoring
$1+T_d^4$ over its roots $\pm z_0,\pm\bar z_0$ gives
\begin{equation}\label{eq:mu8}
 \mu_d=4M_d(z_0)-\log2.
\end{equation}
To see this, note that $m(1+x_\nu)=0$. The substitution $x_1\mapsto x_1^{-1}$
in the torus integral gives $M_d(-z)=M_d(z)$, and conjugation gives
$M_d(\bar z)=M_d(z)$ because $T_d$ has rational coefficients.
Thus the four factors have equal measure.

Write $\chi_8=(8/\cdot)$ for the primitive even quadratic character
of conductor $8$. Grouping the odd terms of the polylogarithmic
series by their residue classes modulo $8$ gives
\[
 A_n(z_0)=\sqrt2\bigl(L(\chi_8,n)+iL(\chi_{-8},n)\bigr)
 \qquad(n\geq1).
\]
At $n=1$ the series converge by Dirichlet's test. Write
$B_n=B_n(z_0)$, so that
\[
 B_n=\begin{cases}
 \sqrt2\,L(\chi_{-8},n)/\pi^{n-1},&n\text{ even},\\
 \sqrt2\,L(\chi_8,n)/\pi^{n-1},&n\text{ odd}.
 \end{cases}
\]
In particular, $B_1=\log(1+\sqrt2)$. Normalize the polylogarithmic
combinations in the triangular formulas by setting $w_1=B_1$ and
\[
 w_{2j}=\pi^{1-2j}\operatorname{Im}\mathcal V_{2j}(z_0),\qquad
 w_{2j+1}=\pi^{-2j}\operatorname{Re}\mathcal U_{2j+1}(z_0)
 \qquad(j\geq1).
\]
Since $-\Log(-iz_0)=i\pi/4$ and $-\Log z_0=-i\pi/4$, both
parities give
\[
 w_n=\sum_{h=0}^{n-1}
       \frac{(-1)^{\lfloor h/2\rfloor}}{4^h h!}B_{n-h}.
\]
Summing over $n$ gives the formal identity
\begin{equation}\label{eq:W8}
 W(t):=\sum_{n\geq1}w_nt^{n-1}
 =\bigl(\cos(t/4)+\sin(t/4)\bigr)
       \sum_{n\geq1}B_nt^{n-1}.
\end{equation}

Substituting $z=z_0$ in Proposition~\ref{higher:triangular} and
using \eqref{eq:mu8} gives, after simplifying the elementary terms,
\begin{equation}\label{eq:mu8-even}
 \mu_{2r}=4\sum_{j=1}^r c_{rj}w_{2j},\qquad
 \mu_{2r+1}=2B_1+4\sum_{j=1}^r d_{rj}w_{2j+1},
\end{equation}
where the sum in the second identity is interpreted as zero when $r=0$.
The coefficients in Proposition~\ref{higher:triangular} satisfy
\begin{equation}\label{higher:c-generating}
 \sum_{r\ge j}c_{rj}s^{2r-1}
 =\frac{2^{2j-2}(\arcsin s)^{2j-1}}{1-s^2},\qquad
 \sum_{r\ge j}d_{rj}s^{2r}
 =\frac{2^{2j-1}(\arcsin s)^{2j}}{1-s^2}.
\end{equation}
To prove these identities, use the product formulas of
\cite[Theorem 17]{Lalin06}, which give
\[
 \sum_{r\ge1}E_r(X)s^{2r-2}
 =\frac{2\cosh(2X\arcsin s)}{\sqrt{1-s^2}},\qquad
 \sum_{r\ge1}O_r(X)s^{2r-1}
 =\frac{2\sinh(2X\arcsin s)}{\sqrt{1-s^2}}.
\]
Interpreting the series as formal differential operators, apply
these identities with $X=\partial_u$ to the even function
$u^{2j-1}/\sin u$ and the odd function $u^{2j}/\sin u$, respectively,
and evaluate at $u=0$. Formula~\eqref{higher:coefficients} and
$\sin(2\arcsin s)=2s\sqrt{1-s^2}$ give \eqref{higher:c-generating}.

Using these generating functions to sum \eqref{eq:mu8-even} yields
\[
 \sum_{d\geq1}\mu_ds^{d-1}
 =\frac{2}{1-s^2}W(2\arcsin s).
\]
Substitute $s=\sin(2t)$ and use \eqref{eq:W8}. We obtain
\begin{equation}\label{eq:conductor8-generating}
 \cos(2t)(\cos t-\sin t)
       \sum_{d\geq1}\mu_d\sin^{d-1}(2t)
 =2\sum_{n\geq1}4^{n-1}B_nt^{n-1}.
\end{equation}
Take the derivative of order $2k-1$ at zero. Only the terms with
$d\leq2k$ contribute on the left, and their coefficients are the
integers $a_{k,d}$ defined in \eqref{eq:a8}. Hence
\[
 m(R_{8,2k})=\sum_{d=1}^{2k}a_{k,d}\mu_d
 =2\cdot4^{2k-1}(2k-1)!B_{2k}.
\]
This proves the first equality in \eqref{eq:R8-measure}; the second
follows from \eqref{higher:functional} with $f=8$.
\end{proof}
\begin{example}
For $k=2$, the exponent vector is $(13,-38,-24,48)$, so
\begin{equation}\label{eq:R8-four}
 m\left(\frac{F_1^{13}F_4^{48}}{F_2^{38}F_3^{24}}\right)
 =\frac{768\sqrt2}{\pi^3}L(\chi_{-8},4)
 =-2L'(\chi_{-8},-3).
\end{equation}
The lower-weight cancellation is explicit:
\[
\begin{aligned}
 \mu_1&=2B_1,&
 \mu_2&=B_1+4B_2,\\
 \mu_3&=\tfrac74B_1+2B_2+8B_3,&
 \mu_4&=\tfrac98B_1+\tfrac{25}{6}B_2+4B_3+16B_4.
\end{aligned}
\]
Thus $13\mu_1-38\mu_2-24\mu_3+48\mu_4=768B_4$.
For $k=3$, the same construction gives the six-variable identity
\begin{equation}\label{eq:R8-six}
 m\left(\frac{F_2^{842}F_3^{1680}F_6^{3840}}
                   {F_1^{121}F_4^{4320}F_5^{1920}}\right)
 =\frac{245760\sqrt2}{\pi^5}L(\chi_{-8},6)
 =2L'(\chi_{-8},-5).
\end{equation}
\end{example}

\subsection{Two constructions at conductor 31}

At conductor $31$, the cycle data reduce the multiplier in
Theorem~\ref{thm:weak} from $8$ to $4$.
Let $\omega=\sqrt{-31}$, with $\operatorname{Im}\omega>0$.
Let $\eta_{31}$ be the $13$-term chain in the data file
\texttt{bloch\_neg31.txt} of \cite{Yasaki}. The file uses the field
generator $(1+\omega)/2$. Since all coefficients of $\eta_{31}$ are even,
taking half the difference with its complex conjugate and collecting
terms gives the integral cycle
\[
  \beta_{31}=\frac12(\eta_{31}-\bar\eta_{31})
            =\sum_{j=1}^8b_j([z_j]-[\bar z_j]).
\]
Here $\gamma_{31}=\eta_{31}-\bar\eta_{31}$ represents
$2\beta_{\rm geo}$ by \cite[Remark 4.9 and Section 6]{BDGRY}.
The values of $z_j,A_j,B_j$, and $b_j$ are listed in
Table~\ref{tab:31-cycle}. Since $\beta_{31}=\gamma_{31}/2$,
equations \eqref{eq:gamma} and \eqref{eq:gamma-regulator} yield
\[
  \delta\beta_{31}=0, \qquad D(\beta_{31})=2\pi d_{31}.
\]

\begin{table}[htbp]
\caption{Data for the Bloch cycle $\beta_{31}$.}
\label{tab:31-cycle}
\centering
$\begin{array}{|c|c|c|c|r|}\hline
 j&z_j&A_j=|z_j|^2&B_j=|1-z_j|^2&b_j\\ \hline
 1&(3+\omega)/4&5/2&2&9\\ \hline
 2&(-3+\omega)/10&2/5&2&6\\ \hline
 3&(11+3\omega)/20&1&9/10&10\\ \hline
 4&(3+\omega)/10&2/5&4/5&9\\ \hline
 5&(19+3\omega)/40&2/5&9/20&12\\ \hline
 6&9(1+\omega)/16&81/8&10&12\\ \hline
 7&(15+\omega)/18&64/81&10/81&6\\ \hline
 8&(11+3\omega)/200&1/100&9/10&2\\ \hline
\end{array}$
\end{table}

For the cycle $\beta_{31}$, Lemma~\ref{lem:correction} gives the
correction integers
\[
  k_2=14,\qquad k_3=48,\qquad k_5=-10,
  \qquad k_p=0\quad(p\ne2,3,5).
\]

The construction in the proof of Theorem~\ref{thm:weak} gives the
following example, with $Q_{A,B}$ as in Section~\ref{sec:cycles}.

\begin{example}
With $A_j,B_j$, and $b_j$ as in Table~\ref{tab:31-cycle}, define
\[
  \widetilde R_{31}(x,y)=\frac{5^{10}}{2^{14}3^{48}}
     \prod_{j=1}^8 Q_{A_j,B_j}(x,y)^{b_j}\in\Q[x,y].
\]
It has total degree $2\sum_j b_j=132$ and satisfies
\[
  m(\widetilde R_{31})=\frac2\pi D(\beta_{31})=4d_{31}.
\]
Numerically, $m(\widetilde R_{31})\approx68.8393544832$.
\end{example}

The polynomial $\widetilde R_{31}$ gives an explicit realization
over $\Q$ of the weak conjecture at conductor $31$.
The numerical identities proposed in \cite[Table 3]{HMR} instead
involve combinations of $d_{31}$ with values at other conductors, namely
$\frac{2}{21}(d_{31}+4d_7)$ and
$\frac{2}{15}(d_{31}+d_{15})$.

For comparison, the general construction gives the following
higher-degree realization.

\begin{example}
Specialize the construction of
Theorem~\ref{thm:weak-higher} to $k=1$ and $f=31$.
Take $M=23$ in \eqref{higher:R-def}. This choice follows directly
from the explicit formulas \eqref{higher:ell-pair} and
\eqref{higher:reference}. To verify this, the remainder estimate
\[
 \left|\operatorname{arctanh}u-u\right|
 \leq \frac{|u|^3}{3(1-u^2)}
 \qquad (|u|<1)
\]
gives
\[
 \frac{|\ell_{31}(s)|}{g_{31}(s)}<0.517
 \quad(0<s\leq0.1),
 \qquad
 \frac{|\ell_{31}(s)|}{g_{31}(s)}<17.83
 \quad(s\geq5),
\]
where in the second estimate we also use
$\ell_{31}(s)=\ell_{31}(1/s)$. A numerical interval evaluation on
$[0.1,5]$, split at the pole $s=1/\sqrt{31}$, gives
\[
 \frac{|\ell_{31}(s)|}{g_{31}(s)}<23
 \qquad\left(s>0,\ s\ne\frac1{\sqrt{31}}\right).
\]
At $s=1.0210852$, a numerical evaluation gives
\[
 22.30502<\frac{|\ell_{31}(s)|}{g_{31}(s)}<22.30503.
\]
Numerically, this is the global maximum, attained at
$s\approx1.0210852009$. Thus $M=23$ is the smallest admissible
integer for this choice of the auxiliary function $V_{31}$.
For the descent step, choose the following parameters.
Let $\zeta=e^{2\pi i/31}$, $c_a=\zeta^a+\zeta^{-a}$, and
$E=\Q(c_1)$. 
Write $\chi=\chi_{-31}$. For $1\le a\le15$, let $r_a$
be the representative of $2a$ modulo $31$ in
$\{-15,\ldots,15\}$, and set $e_a=\chi(a)r_a$. Thus
\[
(e_1,\ldots,e_{15})
=(2,4,-6,8,10,-12,14,-15,-13,-11,9,7,5,-3,1).
\]
Specializing \eqref{higher:even-triangle} to weight $2$ and using
the one-variable Jensen formula, Proposition~\ref{higher:lift},
and the quadratic Gauss sum gives
\[
31m(\mathcal R_{31,2})
=
6d_{31}
-\sum_{a=1}^{15}e_a
 \log\left(\frac{2+c_a}{2-c_a}\right).
\]
To cancel this correction, let
\[
 \lambda=\prod_{a=1}^{15}\left(\frac{2+c_a}{2-c_a}\right)^{e_a},
 \qquad b=\bigl(24(2-c_1)\bigr)^{-84},\qquad
 H(x)=\frac{\Norm_{E/\Q}(x-b\lambda)}{\Norm_{E/\Q}(x-b)}.
\]
Here the factor $24$ separates the distinguished embedding from the
remaining embeddings: $24(2-c_1)<1$, whereas
$24(2-\sigma(c_1))>1$ for every other real embedding $\sigma$.
The exponent $84$ is chosen large enough that this separation is
preserved after multiplication by $\lambda$.
For the chosen inclusion $\sigma_1:E\hookrightarrow\R$, both
$\sigma_1(b)$ and $\sigma_1(b\lambda)$ exceed $1$; at every other
real embedding they lie in $(0,1)$. Jensen's formula therefore gives
$m(H)=\log(b\lambda)-\log b=\log\lambda$, canceling the correction.
Thus an explicit rational function is
\begin{equation}\label{eq:31-analytic}
 \widehat R_{31}(x,y)=\mathcal R_{31,2}(x,y)^{31}H(x),
 \qquad m(\widehat R_{31})=6d_{31}.
\end{equation}

In reduced form, the numerator and denominator of
$\mathcal R_{31,2}$ have total degrees $2M+32=78$ and
$4M+30=122$, respectively. Since the numerator and denominator of
$H$ both have degree $[E:\Q]=15$, those of $\widehat R_{31}$ have
total degrees
\[
 31\cdot78+15=2\,433,
 \qquad
 31\cdot122+15=3\,797,
\]
respectively.
For these explicit choices, the Bloch-cycle construction yields a
substantially lower degree.
\end{example}

\section*{Acknowledgements}
The first author is supported by the National Natural Science
Foundation of China (Grant No. 12231009). The second author is supported by the Natural Science Foundation of Anhui Province (Grant No. 2508085QA017).

\end{document}